\documentclass[twoside,english]{elsarticle}
\usepackage[T1]{fontenc}
\usepackage[latin9]{inputenc}
\usepackage[letterpaper]{geometry}
\usepackage{prettyref}
\usepackage{amsthm}
\usepackage{amsmath}
\usepackage{amssymb}
\usepackage[all]{xy}

\makeatletter
\theoremstyle{plain}
\newtheorem{thm}{Theorem}
\theoremstyle{plain}
\newtheorem{lem}[thm]{Lemma}
 \ifx\proof\undefined\
   \newenvironment{proof}[1][\proofname]{\par
     \normalfont\topsep6\p@\@plus6\p@\relax
     \trivlist
     \itemindent\parindent
     \item[\hskip\labelsep
           \scshape
       #1]\ignorespaces
   }{%
     \endtrivlist\@endpefalse
   }
   \providecommand{\proofname}{Proof}
 \fi
\theoremstyle{plain}
\newtheorem{prop}[thm]{Proposition}
\theoremstyle{plain}
\newtheorem{cor}[thm]{Corollary}

\journal{Arxiv}

\makeatother

\usepackage{babel}

\begin{document}
\global\long\def\H{\ensuremath{\mathrm{H}}}

\global\long\def\res{\ensuremath{\mathrm{res}}}

\global\long\def\ind{\ensuremath{\mathrm{ind}}}

\global\long\def\lcm{\ensuremath{\mathrm{lcm}}}

\global\long\def\sb{\ensuremath{\mathrm{SB}}}

\global\long\def\gl{\ensuremath{\mathrm{GL}}}

\global\long\def\spec{\ensuremath{\mathrm{Spec}}}

\global\long\def\mat{\ensuremath{\mathrm{Mat}}}

\global\long\def\br{\ensuremath{\mathrm{Br}}}

\global\long\def\P{\ensuremath{\mathbb{P}}}

\global\long\def\A{\ensuremath{\mathbb{A}}}

\global\long\def\Z{\ensuremath{\mathbb{Z}}}

\global\long\def\chara{\ensuremath{\mathrm{char}}}

\global\long\def\ram{\ensuremath{\mathrm{ram}}}

\title{Indecomposable and Noncrossed Product Division Algebras over Curves
over Complete Discrete Valuation Rings }

\author[rvt]{Feng Chen}

\ead{fchen@emory.edu}

\address{Dept. of Math \& CS, Emory University, Atlanta GA 30322}
\begin{abstract}
Let $T$ be a complete discrete valuation ring and $\hat{X}$ a smooth
projective curve over $S=\spec(T)$ with closed fibre $X$. Denote
by $F$ the function field of $\hat{X}$ and by $\hat{F}$ the completion
of $F$ with respect to the discrete valuation defined by $X$, the
closed fibre. In this paper, we construct indecomposable and noncrossed
product division algebras over $F$. This is done by defining an index
preserving group homomorphism $s:\br(\hat{F})'\to\br(F)'$, and using
it to lift indecomposable and noncrossed product division algebras
over $\hat{F}$. 
\end{abstract}
\maketitle

\section{Introduction}

Let $\hat{X}$ be a smooth projective curve over $S=\spec(T)$, where
$T$ is a complete discrete valuation ring with uniformizer $t$.
Let $F=K(\hat{X})$ be the function field, and let $\hat{F}=\hat{K}(\hat{X})$
be the completion with respect to the discrete valuation defined by
the closed fibre $X$. We define an index-preserving homomorphism

\[
\br(\hat{F})'\to\br(F)'\]
 that splits the restriction map $\res:\br(F)'\to\br(\hat{F})'$.
Here $\br(-)$ denotes the Brauer group of $-$ and the {}``prime''
denotes the union of the $n$-torsion part of $\br(F)$, where $n$
is prime to the characteristic of $k$, the residue field of $T$.
Using the method of \citet{brussel1995noncrossed} and \citet{brussel1996decomposability},
we can construct indecomposable and noncrossed product division algebras
over $\hat{F}$, and lift these constructions to $F$ using our homomorphism,
generalizing the constructions in \citet{brussel2009indecomposable},
where indecomposable and noncrossed product division algebras over
function fields of $p$-adic curves are constructed. 

Recall that if $K$ is a field, a $K$-\emph{division algebra} $D$
is a division ring that is finite-dimensional and central over $K$.
The\emph{ period }or \emph{exponent }of $D$ is the order of the class
$[D]$ in $\br(K)$, and the \emph{index }of $D$ is the square root
of $D$'s $K$-dimension. A \emph{noncrossed product }is a $K$-division
algebra whose structure is not given by a Galois 2-cocycle. Noncrossed
products were first constructed by \citet{amitsur1972oncentral},
settling a longstanding open problem. Since then there have been several
other constructions, including \citet{saltman1978noncrossed}, \citet{jacob1986anew},
\citet{brussel1995noncrossed}. 

A $K$-division algebra is $ $ \emph{indecomposable }if it cannot
be expressed as the tensor product of two nontrivial $K$-division
algebras. It is easy to see that all division algebras of period not
a prime power are decomposable, so the problem of producing an indecomposable
division algebra is only interesting when the period and index are
unequal prime powers. Therefore we will only consider division algebras
of prime power period and index in this paper. Then it is not hard
to see that all division algebras of equal (prime power) period and
index are trivially indecomposable. Albert constructed decomposable
division algebras in the 1930's, but indecomposable division algebras
of unequal (2-power) period and index did not appear until \citet{saltman1979indecomposable}and
\citet{amitsur1979division}. Since then there have been several constructions,
including \citet{tignol1987algvcbres}, \citet{jacob1990division},
\citet{jacob1991indecomposible}, \citet{schofield1992theindex},
\citet{brussel1996decomposability} and \citet{mckinnie2008indecomposable}.

It is the author's pleasure to thank Prof. Brussel, his thesis adviser.
The author is greatly indebted to him for his patience and suggestions
during the preparation of the paper. The author would also like to
thank Prof. Suresh and Prof. Parimala for many instructive discussions
and their most valuable comments and critiques. Finally the author
thanks Prof. Harbater for reading a first draft of the paper and his
valuable suggestions and comments to improve the writing.

\section{Patching over Fields \label{sec:Patching-Over-Fields}}

Our construction is based on the method of \emph{patching over fields}
introduced in \citet{harbater2007patching}. In this section, we will
recall this method. Throughout this section, $T$ will be a complete
discrete valuation ring with uniformizer $t$, fraction field $K$
and residue field $k$. Let $\hat{X}$ be a smooth projective $T$-curve
with function field $F$ such that the reduced irreducible components
of its closed fibre $X$ is regular. (Given $F$, such an $\hat{X}$
always exists by resolution of singularities; cf. \citet{abhyankar1969resolution}
or \citet{lipman1975introduction}). Let $f:\hat{X}\to\P_{T}^{1}$
be a finite morphism such that the inverse image $S$ of $\infty\in\P_{k}^{1}$
contains all the points of $X$ at which distinct irreducible components
meet. (Such a morphism exists by \citet[Proposition 6.6]{harbater2007patching}).
We will call $(\hat{X},S)$ a \emph{regular} $T$-\emph{model} of
$F$. 

We follow \citet[Section 6]{harbater2007patching} to introduce the
notation. Given an irreducible component $X_{0}$ of $X$ with generic
point $\eta$, consider the local ring of $\hat{X}$ at $\eta$. For
a (possibly empty) proper subset $U$ of $X_{0}$, we let $R_{U}$
denote the subring of this local ring consisting of rational functions
that are regular at each point of $U$. In particular, $R_{\phi}$
is the local ring of $\hat{X}$ at the generic point of the component
$X_{0}$. The $t$-adic completion of $R_{U}$ is denoted by $\hat{R}_{U}$.
If $P$ is a closed point of $X$, we write $R_{P}$ for the local
ring of $\hat{X}$ at $P$, and $\hat{R}_{P}$ for its completion
at its maximal ideal. A height 1 prime ideal $\mathfrak{p}$ that
contains $t$ determines a \emph{branch} of $X$ at $P$, i.e., an
irreducible component of the pullback of $X$ to $\spec(\hat{R}_{P})$.
Similarly the contraction of $\mathfrak{p}$ to the local ring of
$\hat{X}$ at $P$ determines an irreducible component $X_{0}$ of
$X$, and we say that $\mathfrak{p}$ \emph{lies on }$X_{0}$. Note
that a branch $\mathfrak{p}$ uniquely determines a closed point $P$
and an irreducible component $X_{0}$. In general, there can be several
branches $\mathfrak{p}$ on $X_{0}$ at a point $P$; but if $X_{0}$
is smooth at $P$ then there is a unique branch $\mathfrak{p}$ on
$X_{0}$ at $ $$P$. We write $\hat{R}_{\mathfrak{p}}$ for the completion
of the localization of $\hat{R}_{P}$ at $\mathfrak{p}$; thus $\hat{R}_{P}$
is contained in $\hat{R}_{\mathfrak{p}},$ which is a complete discrete
valuation ring. 

Since $\hat{X}$ is normal, the local ring $R_{P}$ is integrally
closed and hence unibranched; and since $T$ is a complete discrete
valuation ring, $R_{P}$ is excellent and hence $\hat{R}_{P}$ is
a domain (cf. \citet[Scholie 7.8.3(ii,iii,vii)]{grothendieck1961elements}).
For nonempty $U$ as above and $Q\in U$, $\hat{R}_{U}/t^{n}\hat{R}_{U}\to\hat{R}_{Q}/t^{n}\hat{R}_{Q}$
is injective for all $n$ and hence $\hat{R}_{U}\to\hat{R}_{Q}$ is
also injective. Thus $\hat{R}_{U}$ is also a domain. Note that the
same is true if $U$ is empty. The fraction field s of the domains
$\hat{R}_{U},\hat{R}_{P}$ and $\hat{R}_{\mathfrak{p}}$ will be denoted
by $F_{U},F_{P}$and $F_{\mathfrak{p}}$. 

If $\mathfrak{p}$ is a branch at $P$ lying on the closure of $U\subset X_{0}$,
then there are natural inclusions of $\hat{R}_{P}$ and $\hat{R}_{U}$
into $\hat{R}_{\mathfrak{p}}$, and hence of $F_{P}$ and $F_{U}$
into $F_{\mathfrak{p}}$. The inclusion of $\hat{R}_{P}$ was observed
above; for $\hat{R}_{U}$, note that the localization of $R_{U}$
and of $R_{p}$ at the generic point of $X_{0}$ are the same; and
this localization is naturally contained in the $t$-adically complete
ring $\hat{R}_{\mathfrak{p}}$. Thus so is $R_{U}$ and hence its
$t$-adic completion $\hat{R}_{U}$. 

In the above context, assume $f:\hat{X}\to\P_{T}^{1}$ is a finite
morphism such that $\mathfrak{P}=f^{-1}(\infty)$ contains all points
at which distinct irreducible components of the closed fibre $X\subset\hat{X}$
meet (Such an $f$ always exists by \citet[Proposition 6.6]{harbater2007patching}).
We let $\mathfrak{U}$ be the collection of irreducible components
$U$ of $f^{-1}(\A_{k}^{1})$, and let $\mathfrak{B}$ be the collection
of all branches $\mathfrak{p}$ at all points of $\mathfrak{P}$. 

The inclusions of $\hat{R}_{U}$ and of $\hat{R}_{Q}$ into $\hat{R}_{\mathfrak{p}}$,
for $\mathfrak{p}=(U,Q)$, induce inclusions of the corresponding
fraction fields $F_{U}$ and $F_{Q}$ into the fraction field $F_{\mathfrak{p}}$
of $\hat{R}_{\mathfrak{p}}$. Let $I$ be the index set consisting
of all $U,Q,\mathfrak{p}$ described above. Via the above inclusions,
the collection of all $F_{\xi}$, for $\xi\in I$, then forms an inverse
system with respect to the ordering given by setting $U\succ\mathfrak{p}$
and $Q\succ\mathfrak{p}$ if $\mathfrak{p}=(U,Q)$. 

Under the above hypotheses, suppose that for every field extension
$L$ of $F$, we are given a category $\mathfrak{A}(L)$ of algebraic
structures over $L$(i.e. finite dimensional $L$-vector spaces with
additional structure, e.g. associative $L$-algebras), along with
base-change functors $\mathfrak{A}(L)\to\mathfrak{A}(L')$ when $L\subseteq L'$.
An $\mathfrak{A}$-\emph{patching problem }for $(\hat{X,}S)$ consists
of an object $V_{\xi}$ in $\mathfrak{A}(F_{\xi})$ for each $\xi\in I$,
together with isomorphisms $\phi_{U,\mathfrak{p}}:V_{U}\otimes_{F_{U}}F_{\mathfrak{p}}\to V_{\mathfrak{p}}$
and $\phi_{Q,\mathfrak{p}}:V_{Q}\otimes_{F_{Q}}F_{\mathfrak{p}}\to V_{\mathfrak{p}}$
in $\mathfrak{A}(F_{\mathfrak{p}})$. These patching problems form
a category, denoted by $\mathrm{PP}_{\mathfrak{A}}(\hat{X},S)$, and
there is a base change functor $\mathfrak{A}(F)\to\mathrm{PP}_{\mathfrak{A}}(\hat{X},S)$. 

If an object $V\in\mathfrak{A}(F)$ induces a given patching problem
up to isomorphism, we will say that $V$ is a \emph{solution} to that
patching problem, or that it is \emph{obtained by patching }the objects
$V_{\xi}$. We similarly speak of obtaining a morphism over $F$ by
patching morphisms in $\mathrm{PP}_{\mathfrak{A}}(\hat{X},S)$. The
next result is given by \citet[Theorem 7.2]{harbater2007patching}. 
\begin{thm}
\label{thm:patching-br-gp}Let $T$ be a complete discrete valuation
ring. Let $\hat{X}$ be a smooth connected projective $T$-curve with
closed fibre $X$. Let $U_{1},U_{2}\subseteq X$, let $U_{0}=U_{1}\cap U_{2}$,
and let $F_{i}:=F_{U_{i}}(i=0,1,2)$. Let $U=U_{1}\cup U_{2}$ and
form the fibre product of groups $\br(F_{1})\times_{\br(F_{0})}\br(F_{2})$
with respect to the maps $\br(F_{i})\to\br(F_{0})$ induced by $F_{i}\hookrightarrow F_{0}$.
Then the base change map $\beta:\br(F_{_{U}})\to\br(F_{1})\times_{\br(F_{0})}\br(F_{2})$
is a group isomorphism. 
\end{thm}
The above Theorem says that giving a Brauer class over a function
field $F$ is equivalent to giving compatible division algebras over
the patches. The nice thing about patching Brauer classes over a function
field $F$ is that we have good control of the index, which is stated
in \citet[Theorem 5.1]{harbater2009applications}. 
\begin{thm}
\label{thm:hhk-index}Under the above notation, let $A$ be a central
simple $F$-algebra. Then $\ind(A)=\lcm_{\xi\in\mathfrak{P}\cup\mathfrak{U}}(\ind(A_{F_{\xi}}))$. 
\end{thm}
To conclude this section, we record a variant of Hensel's Lemma from
\citet[Lemma 4.5]{harbater2009applications} that will be used over
and over again in the index computation. 
\begin{lem}
\label{lem:hhk-hensel}Let $R$ be a ring and $I$ an ideal such that
$R$ is $I$-adically complete. Let $X$ be an affine $R$-scheme
with structure morphism $\phi:X\to\spec R$. Let $n\ge0$. If $s_{n}:\spec(R/I^{n})\to X\times_{R}(R/I^{n})$
is a section of $\phi_{n}:\phi\times_{R}(R/I^{n})$ and its image
lies in the smooth locus of $\phi$, then $s_{n}$ may be extended
to a section of $\phi$. 
\end{lem}

\section{Splitting Map\label{sec:Splitting-Map}}

Let $T$ be a complete discrete valuation ring with uniformizer $t$
and residue field $k$. By a smooth curve $\hat{X}$ over $T$, we
will mean a scheme $\hat{X}$ which is projective and smooth of relative
dimension $1$ over $\spec(T)$. In particular, $\hat{X}$ is flat
and of finite presentation over $\spec(T)$. Let $F=K(\hat{X})$ be
the function field of $\hat{X}$. Note that since $\hat{X}$ is smooth,
the closed fibre $X$ is smooth, integral, connected and of codimension
1, hence determines a discrete valuation ring on $F$. Let $\hat{F}=\hat{K}(\hat{X})$
be the completion of $F$ with respect to this discrete valuation.
Throughout the paper, $n$ will denote an integer which is prime to
the characteristic of $k$. 

We will be using the following notation for cohomology groups in the
sequel: For an integer $r$, we let 

$\mu_{n}^{r}=\begin{cases}
\mu_{n}^{\otimes r} & \textrm{for }r\ge0,\\
\hom(\mu_{n}^{\otimes-r},\mu_{n}) & \textrm{for }r<0.\end{cases}$

For a fixed integer $n$, and for any field $K$, we will let $\H^{q}(K,r)=\H^{q}(K,\mu_{n}^{\otimes r})$
and $\H^{q}(K)=\H^{q}(K,q-1)=\H^{q}(K,\mu_{n}^{\otimes q-1})$. In
particular, $\H^{2}(K)=\:_{n}\br(K)$ will be the $n$-torsion part
of the Brauer group of $K$; and $\H^{1}(K)$ will be the $n$-torsion
part of the character group of $K$. 

Adopting the above notation, in this section we will define a map
$s:\H^{2}(\hat{F)}\to\H^{2}(F)$ and show that $s$ has the following
properties:
\begin{itemize}
\item $s$ is a group homomorphism;
\item $s$ splits the restriction;
\item $s$ preserves index of Brauer classes.
\end{itemize}
Once such a map $s$ is defined, we could use it to construct indecomposable
division algebras and noncrossed product division algebras over $F$,
as in section \ref{sec:Indecomposable-and-Noncrossed-Product}.

\subsection{Construction over an Open Affine Subset\label{sub:Construction-over-Open-Affine}}

Given an element $\hat{\gamma}\in\H^{2}(\hat{F})$, we will define
a lift $\gamma_{U}$ to $F_{U}$ of $\hat{\gamma}$. Note that since
$\hat{F}$ is a complete discretely valued field with $t$ a unifomizer,
and with $k(X)$ the residue field. We have an exact Witt Sequence
as in \citet[II.7.10 and II.7.11]{garibaldi2003cohomological},

\begin{equation}
0\to\H^{2}(k(X))\to\H^{2}(\hat{F})\to\H^{1}(k(X))\to0\label{eq:witt-exact}\end{equation}
split (non-canonically) by the cup product with $(t)\in\H^{1}(k(X))$.
Hence each element $\hat{\gamma}\in\H^{2}(\hat{F})$ can be written
as a sum $\gamma_{0}+(\chi_{0},t)$, with $\gamma_{0}\in\H^{2}(k(X))$
and $\chi_{0}\in\H^{1}(k(X))$ (Note that here we are identifying
$\H^{r}(k(X))$ as a subgroup of $\H^{r}(\hat{F})$, for $r=1,2$,
as in \citet[II.7.10 and II.7.11]{garibaldi2003cohomological}). Here
we use the notation $(\chi_{0},t)$ to denote the cup product $\chi_{0}\cup(t)$,
and we will use this notation throughout the paper without further
explanation.

Let $U$ be an open affine subset of $X$ so that neither $\gamma_{0}$
nor $\chi_{0}$ ramifies at any closed point of $U$. This implies
that $\gamma_{0}\in\H^{2}(k[U])$ and $\chi_{0}\in\H^{1}(k[U])$ by
purity(cf, \citet{colliot-thtextbackslasheltextbackslash`ene1992birational}),
where $k[U]$ denotes the ring of regular functions of the affine
scheme $U$. 

By \citet{cipolla1977remarks}, there exists a canonical isomorphism
$\H^{2}(\hat{R}_{U})\to\H^{2}(k[U])$ since $\hat{R}_{U}$ is $t$-adically
complete and $k[U]\cong\hat{R}_{U}/(t)$; therefore there is a unique
lift of $\gamma_{0}$ to $\H^{2}(\hat{R}_{U})$. At the same time,
\citet[Th\'eor\`em 8.3]{grothendieck2002revtextbackslashtextasciicircumetements}
implies that there is a unique lift of $\chi_{0}$ to $\H^{1}(\hat{R}_{U})$
as well. Taking $\tilde{\gamma}_{0}$ and $\tilde{\chi}_{0}$ as the
lifts of $\gamma_{0}$ and $\chi_{0}$ to $\hat{R}_{U}$, we will
let 

\begin{equation}
\gamma_{U}=\tilde{\gamma}_{0}+(\tilde{\chi}_{0},t)\end{equation}
be the lift of $\hat{\gamma}$ to $\H^{2}(F_{U})$.

\subsection{Construction over Closed Points\label{sub:Construction-over-Closed}}

Fix an open affine subset $U$ of $X$ and let $\mathfrak{P}=X\backslash U$.
In order to apply the patching result we recalled in \ref{sec:Patching-Over-Fields},
we need to define a $\gamma_{P}$ for each $P\in\mathfrak{P}$ in
such a way that when $\mathfrak{p}=(U,P)$ is the unique branch of
$U$ at $P$, the restriction to $F_{\mathfrak{p}}$ of $\gamma_{P}$
and $\gamma_{U}$ agree with each other, i.e., $\res_{F_{\mathfrak{p}}}(\gamma_{P})=\res_{F_{\mathfrak{p}}}(\gamma_{U})$
(Recall there are field embeddings $F_{P}\hookrightarrow F_{\mathfrak{p}}$
and $F_{U}\hookrightarrow F_{\mathfrak{p}}$ for $\mathfrak{p}=(U,P),$
as in Section \ref{sec:Patching-Over-Fields}, hence there are restrictions
$\res:\H^{2}(F_{U})\to\H^{2}(F_{\mathfrak{p}})$ and $\res:\H^{2}(F_{P})\to\H^{2}(F_{\mathfrak{p}})$.
For more details on these restriction maps, see \citet{serre1979localfields}).

Note that since $\hat{X}$ is regular and the closed fibre $X$ is
smooth, the maximal ideal of the local ring $R_{P}$ is generated
by two generators, $t$ and $\pi$. So is $\hat{R}_{P}$. 

We define $\gamma_{P}$ in the following way: There is a field embedding
$F_{U}\to F_{\mathfrak{p}}$, hence a canonical restriction $\res:\H^{2}(F_{U})\to\H^{2}(F_{\mathfrak{p}})$.
Let $\gamma_{\mathfrak{p}}$ be the image of $\gamma_{U}$ under this
restriction. Observe that $F_{\mathfrak{p}}$ is a complete discretely
valued field with residue field $\kappa(\mathfrak{p})$; furthermore,
$\kappa(\mathfrak{p)}$ is also a complete discretely valued field
with residue field $\kappa(P)$. Therefore, applying \citet[II.7.10 and II.7.11]{garibaldi2003cohomological}
twice, we get the following decomposition of $\H^{2}(F_{\mathfrak{p}})$:
\begin{equation}
\H^{2}(F_{\mathfrak{p}})\cong\H^{2}(\kappa(P))\oplus\H^{1}(\kappa(P))\oplus\H^{1}(\kappa(P))\oplus\H^{0}(\kappa(P)).\end{equation}
In other words, each element $\gamma_{\mathfrak{p}}\in\H^{2}(F_{\mathfrak{p}})$
can be written as $\gamma_{\mathfrak{p}}=\gamma_{0,0}+(\chi_{1,}\pi)+(\chi_{2}+(\pi^{r}),t)$,
where $\gamma_{0,0}\in\H^{2}(\kappa(P)),\chi_{1,}\chi_{2}\in\H^{1}(\kappa(P)),r\in\H^{0}(\kappa(P))\cong\Z/n\Z$
and $(\pi^{r})$ denote the image in $\H^{1}(\kappa(P$)) of $\pi^{r}$
under the Kummer map.  Note that by our notation, $\H^{0}(\kappa(P))=\H^{0}(\kappa(P),\mu_{n}^{-1})=\Z/n\Z$. 

In order to define a lift for $\gamma_{\mathfrak{p}}$ to $F_{P}$,
we first show that all characters in $\H^{1}(\kappa(\mathfrak{p}))$
can be lifted by proving the following lemma.
\begin{lem}
\label{lem:lift-characters}Let $\chi\in\H^{1}(\kappa(\mathfrak{p}))$
be a character. Then there is a unique $\tilde{\chi}\in\H^{1}(F_{P})$
that lifts $\chi$.\end{lem}
\begin{proof}
Since $\kappa(\mathfrak{p})$ is a complete discretely valued field
with residue field $\kappa(P)$, we have the classical Witt's decomposition
for $\chi$, 

\[
\chi=\chi_{0}+(\pi^{r}),\]
 where $\chi_{0}\in\H^{1}(\kappa(P))$ and $r\in\H^{0}(\kappa(P$)).
Note that $\chi_{0}$ can be lifted without any difficulty by \citet[Th\'eor\`em 8.3]{grothendieck2002revtextbackslashtextasciicircumetements};
the only trouble comes from $(\pi^{r})$. 

Let $L,L_{0}/\kappa(\mathfrak{p})$ be the field extension determined
by $\chi,\chi_{0}$ respectively. Then $L_{0}$ is the maximal unramified
subextension of $\kappa(\mathfrak{p})$ inside $L$ and $L/L_{0}$
is a totally ramified extension determined by the character $(\pi^{r})$.
Now \citet[Theorem II.3.5]{fesenko2002localfields} implies that $(\pi^{r})$
can be lifted to $\H^{1}(F_{P})$ in a unique fashion as well, since
$\kappa(\mathfrak{p})$ is a complete discretely valued field. 
\end{proof}
Now we are ready to define a lift for $\hat{\gamma}$ in $\H^{2}(F_{P})$.
Again \citet{cipolla1977remarks} implies that $\H^{2}(\kappa(P))\cong\H^{2}(\hat{R}_{P})$
and Lemma \ref{lem:lift-characters} implies that $\chi_{1},\chi_{2}+(\pi^{r})$
can be lifted to $\H^{1}(\hat{R}_{P})$ uniquely. Hence each component
of $\H^{2}(F_{\mathfrak{p}})$ can be lifted to $\hat{R}_{P}$, and
thus we will set 

\begin{equation}
\gamma_{P}=\tilde{\gamma}_{0,0}+(\tilde{\chi}_{1},\pi)+(\tilde{\chi}{}_{2}+(\pi^{r}),t).\end{equation}
where $\tilde{\gamma},\tilde{\chi}_{1},\tilde{\chi}_{2}$ are the
lifts of $\gamma_{0,0,}\chi_{1},\chi_{2}$ to $\hat{R}_{P}$ (and
hence to $F_{P}$), respectively. Therefore this $\gamma_{P}$ is
a unique lift of $\gamma_{\mathfrak{p}}$ to $F_{P}$. The assignment
of $s_{P}(\gamma_{\mathfrak{p}})=s_{P}$ will yield a map $s_{P}:\H^{2}(F_{\mathfrak{p}})\to\H^{2}(F_{P})$.
It is not hard to see that $s_{P}$ is a group homomorphism, since
it is a group homomorphism on each of the components.

\subsection{The Map is Well Defined}

In this section we show that $\gamma_{U}$ and $\gamma_{P}$ that
we constructed in Section \ref{sub:Construction-over-Open-Affine}
and Section\ref{sub:Construction-over-Closed} are compatible in the
sense of patching, that is $\res_{F_{\mathfrak{p}}}(\gamma_{U})=\res_{F_{\mathfrak{p}}}(\gamma_{P})$
for each $P\in\mathfrak{P}=X\backslash U$ when $\mathfrak{p}=(U,P)$
is the unique branch of $U$ at $P$. 

We claim that the compatibility will be proved if we can show that
$s_{P}$ splits the restriction map $\res_{F_{\mathfrak{p}}}:\H^{2}(F_{P})\to\H^{2}(F_{\mathfrak{p}})$,
or equivalently, $\res_{F_{\mathfrak{p}}}\circ s_{P}$ is the identity
map. This is true because $\gamma_{P}=s_{P}(\gamma_{\mathfrak{p}})=s_{P}\circ\res_{F_{\mathfrak{p}}}(\gamma_{U})$,
hence we would have that $\res_{F_{\mathfrak{p}}}(\gamma_{P})=\res_{F_{\mathfrak{p}}}(\gamma_{U})$
if $\res_{F_{\mathfrak{p}}}\circ s_{P}$ is the identity map. So it
suffices to prove the following
\begin{prop}
$s_{P}$ as defined in \ref{sub:Construction-over-Closed} splits
the restriction $\res:\H^{2}(F_{P})\to\H^{2}(F_{\mathfrak{p}})$,
that is, $\res\circ s_{P}$ is the identity map. \end{prop}
\begin{proof}
Take an arbitrary element $\gamma_{\mathfrak{p}}\in\H^{2}(F_{\mathfrak{p}})$.
As in section \ref{sub:Construction-over-Closed}, we write $\gamma_{\mathfrak{p}}=\gamma_{0,0}+(\chi_{1,}\pi)+(\chi_{2}+(\pi^{r}),t)$.
Therefore it is easily checked that

\begin{eqnarray*}
\res\circ s_{P}(\gamma_{\mathfrak{p}}) & = & \res\circ s_{P}(\gamma_{0,0}+(\chi_{1},\pi)+(\chi_{2}+(\pi^{r}),t))\\
 & = & \res(\tilde{\gamma}_{0,0}+(\tilde{\chi}_{1},\pi)+(\tilde{\chi}_{2}+(\pi^{r}),t))\\
 & = & \gamma_{0,0}+(\chi_{1},\pi)+(\chi_{2}+(\pi^{r}),t)\\
 & = & \gamma_{\mathfrak{p}.}\end{eqnarray*}

\end{proof}
Thus $\gamma_{U},\gamma_{P}$ will patch and yield $\gamma\in\H^{2}(F)$,
by \citet[Theorem 7.2]{harbater2007patching}. But there is one more
thing we have to check before we can say we have a map $s:\H^{2}(\hat{F})\to\H^{2}(F)$:
we need to show that $\gamma$ is independent of the choice of the
open affine subset $U$ of $X$. In order to do this, we prove the
following
\begin{lem}
\label{lem:Indep-open-affine}Let $T$ be a complete discrete valuation
ring with residue field $k$; let $\hat{X}$ be a smooth projective
$T$-curve with function field $F$ and closed fibre $X$. Let $\hat{F}$
be the completion of $F$ with respect to the discrete valuation induced
by $X$, and denote by $k(X)$ the corresponding residue field. Take
an element $\hat{\gamma}=\gamma_{0}+(\chi_{0},t)\in\H^{2}(\hat{F})$,
where $\gamma_{0}\in\H^{2}(k(X))$ and $\chi_{0}\in\H^{1}(k(X))$.
Assume that $U_{1},U_{2}$ are two open affine subsets of $X$ so
that neither $\gamma_{0},\chi_{0}$ is ramified on any point of $U_{1}\cup U_{2}$.
Let $\mathfrak{P_{1}},\mathfrak{P}_{2}$ be the complements of $U_{1},U_{2}$
respectively. We construct two Brauer classes $\gamma,\gamma'\in\H^{2}(F)$
by patching as we did above, while using $U_{1}$ and $U_{2}$ as
the open affine subset in the construction, respectively. Then $\gamma,\gamma'$
denote the same Brauer class in $\H^{2}(F)$. \end{lem}
\begin{proof}
We first deal with the case where $U_{1}$ is contained in $U_{2}$.
In this case we have a field embedding $F_{U_{2}}\hookrightarrow F_{U_{1}}$.
Let $\gamma_{i}$ be the lift of $\gamma_{0}$ to $\H^{2}(F_{U_{i}})$,
we must have $\gamma_{1}=\res_{F_{U_{1}}}(\gamma_{2})$, since both
$\gamma_{1}$ and $\gamma_{2}$ are the image of $\gamma_{0}$; in
other words, $\res_{F_{U_{2}}}(\gamma)=\res_{F_{U_{2}}}(\gamma')$.
By the construction in Section \ref{sub:Construction-over-Closed},
it follows that for every $P\in\mathfrak{P}_{2}$, $\res_{F_{P}}(\gamma)=\res_{F_{P}}(\gamma')$.
Therefore it follows that $\gamma=\gamma'$, by \citet[Theorem 7.2]{harbater2007patching}.
This proves the Lemma in the case where $U_{1}$ is contained in $U_{2}$. 

In the general case, let $U_{3}$ be an open affine subset of $U_{1}\cap U_{2}$.
Clearly $\gamma_{0}$ and $\chi_{0}$ are both unramified at every
point of $U_{3}$. Let $\gamma''\in\H^{2}(F)$ be the Brauer class
constructed by patching as above, using $U_{3}$ as the open affine
subset in the construction. It follows that $\gamma''=\gamma$ and
$\gamma''=\gamma'$ since $U_{3}$ is contained in both $U_{1}$ and
$U_{2}$, by what we just proved for the case where one open affine
subset is contained in the other. Hence $\gamma=\gamma'=\gamma''\in\H^{2}(F)$,
which proves the Lemma in the general case. 
\end{proof}

\subsection{$s$ Splits the Restriction Map\label{sub:s-Splits-restriction}}

Recall the notation: let $T$ be a complete discrete valuation ring
with residue field $k$ and uniformizer $t$. Let $\hat{X}$ be a
smooth projective $T$-curve with function field $F$ and closed fibre
$X$. Let $\hat{F}$ be the completion of $F$ with respect to the
discrete valuation induced by $X$. Let $s:\H^{2}(\hat{F})\to\H^{2}(F)$
be the map defined by patching as in section \ref{sub:Construction-over-Open-Affine}
and section \ref{sub:Construction-over-Closed}. We will show that
$s$ splits the restriction map $\res:\H^{2}(F)\to\H^{2}(\hat{F}$).
Hence index of Brauer classes cannot go up under the map $s$, because
restriction can never raise index. In particular, we prove the following
Proposition.
\begin{prop}
\label{pro:s-is-section}The map $s$ is a section to the restriction
map $\res_{\hat{F}}:\H^{2}(F)\to\H^{2}(\hat{F})$. \end{prop}
\begin{proof}
It suffices to show that $\res\circ s$ is the identity map on $\H^{2}(\hat{F})$.
Since $\H^{2}(\hat{F})\cong\H^{2}(k(X))\oplus\H^{1}(k(X))$, it suffices
to show that $\res_{\hat{F}}\circ s$ is the identity map on both
components; that is, given $\hat{\gamma}=\gamma_{0}+(\chi_{0},t)$
where $\gamma_{0}\in\H^{2}(k(X))$ and $\chi_{0}\in\H^{1}(k(X))$,
the Proposition will follow if we can show that $\res_{\hat{F}}\circ s(\gamma_{0})=\gamma_{0}$
and $\res_{\hat{F}}\circ s((\chi_{0},t))=(\chi_{0},t)$. 

Take an open affine subset $U$ of $X$ so that $\gamma_{0},\chi_{0}$
are both unramified on every point of $U$; that is, we have $\gamma_{0}\in\H^{2}(k[U])$
and $\chi_{0}\in\H^{1}(k[U])$. Note that we have the following commutative
diagram (For a field $E$, $\H_{\mathrm{nr}}^{2}(E)$ denotes the
unramified part of $\H_{\mathrm{nr}}^{2}(E)$, or equivalently, $\H_{\mathrm{nr}}^{2}(E)=\cap_{v}\H^{2}(E_{v})$,
where $v$ runs through all discrete valuations on $E$, and $E_{v}$
denotes the completion of $E$ at $v$. See \citet{colliot-thtextbackslasheltextbackslash`ene1992birational}
for more details on the unramified cohomology.): 

$\xymatrix{\H^{2}(k(X))\ar@{->}[r]_{f}^{\sim} & \H_{\text{nr}}^{2}(\hat{F})\ar[d]^{s}\\
\H^{2}(\hat{R}_{U})\ar@{^{(}->}[u]^{g}\ar@{_{(}->}[d]^{h} & \H^{2}(F)\ar[dl]^{\res_{F_{U}}}\\
\H^{2}(F_{U})}
$

The commutativity of the above diagram follows simply from the construction
of $ $over open affine subset we outline in Section \ref{sub:Construction-over-Open-Affine}.
Therefore $\res_{\hat{F}}$ on $s(\gamma_{0})$ is the same as $f\circ g\circ h^{-1}\circ\res_{F_{U}}$,
and thus $\res_{\hat{F}}\circ s(\gamma_{0})=f\circ g\circ h^{-1}\circ\res_{F_{U}}\circ s(\gamma_{0})=\gamma_{0}$.$ $
(Note in fact $h$ has no inverse; however we can find an inverse
image under $h$ for $\res_{F_{U}}\circ s(\gamma_{0})$, so we write
$h^{-1}$ only merely as a shorthand notation here.)

To show that $\res_{\hat{F}}\circ s((\chi_{0},t))=(\chi_{0},t)$,
it suffices to show that $\ram(\res_{\hat{F}}\circ s((\chi_{0},t)))=\chi_{0}$,
where $\ram:\H^{2}(\hat{F})\to\H^{1}($$k(X))$ denotes the ramification
map on $\H^{2}(\hat{F})$ with respect to the valuation determined
by the closed fibre $X$. Since $\chi_{0}\in\H^{1}(k[U])$, we have
$\ram(\res_{\hat{F}}\circ s((\chi_{0},t)))=\ram((\tilde{\chi}_{0},t))$
where $\tilde{\chi}_{0}$ denotes the lift of $\chi_{0}$ to $\H^{1}(\hat{R}_{U})$,
as we did in Section \ref{sub:Construction-over-Open-Affine} (Since
$\H^{1}(\hat{R}_{U})\cong\H^{1}(k[U])$, $\tilde{\chi}_{0}$ can be
viewed as as element of $\H^{1}(k[U])$, and hence element of $\H^{1}(k(X))$
via the injection $\H^{1}(k[U])\hookrightarrow\H^{1}(k(X))$, and
finally element of $\H^{1}(\hat{F})$ via the injection $\H^{1}(k(X))\hookrightarrow\H^{1}(\hat{F})$).
Therefore the image in $\H^{1}(\hat{F})$ of $\tilde{\chi}_{0}$ under
the composition of these maps is in fact $\chi_{0}$, since all these
maps are injective. Then it is easy to see that $\ram((\tilde{\chi}_{0},t))=\tilde{\chi}_{0}=\chi_{0}\in\H^{1}(k[X])$,
as desired.
\end{proof}
The following corollary is immediate: 
\begin{cor}
\label{cor:raise-index}Index of Brauer classes cannot go down under
the map $s$. \end{cor}
\begin{proof}
Take $\hat{\gamma}\in\H^{2}(\hat{F})$ and let $\gamma=s(\hat{\gamma})$.
By Proposition \ref{pro:s-is-section} we must have that $\hat{\gamma}=\res_{\hat{F}}(\gamma)$,
therefore $\ind(\hat{\gamma})|\ind(\gamma)$. This proves that $s$
can never lower index of Brauer classes.
\end{proof}

\section{$s$ Preserves Index of Brauer Classes\label{sec:Index-Computation}}

In this section, we will show that the splitting map $s$ that we
defined in section \ref{sec:Splitting-Map} has one more property
that is crucial to the construction of indecomposable and noncrossed
product division algebras over $p$-adic curves, that is, $s$ preserves
index of Brauer classes. In other words, $\ind(\hat{\gamma})=\ind(\gamma)=\ind(s(\hat{\gamma}))$.
We make the following elementary observation, which is true for Brauer
classes over an arbitrary field. 
\begin{prop}
\label{pro:ele-index}Let $k$ be an arbitrary field. Let $\gamma\in\H^{2}(k)$
be a Brauer class with the following decomposition: $\gamma=\gamma_{0}+(\chi,t)$,
where $\gamma_{0}\in\H^{2}(k)$, $\chi\in\H^{1}(k)$ and $t$ is an
arbitrary element of $k$. Then $\ind(\gamma)|\ind(\gamma_{0,l})\cdot\exp(\chi)$,
where $\gamma_{0,l}$ denotes the base extension of $\gamma_{0}$
to $l/k$, where $l$ is the field extension determined by $\chi$. \end{prop}
\begin{proof}
Let $E/l$ be a minimal extension that splits $\gamma_{0,l}$. Then
$[E:l]=\ind(\gamma_{0,l})$. Also there is some $E'/k$ with $[E':k]=\exp(\chi)$
which splits $\chi$ and hence $(\chi,t)$; therefore $EE'$ will
split $\gamma$, furthermore it is not hard to see that $[EE':k]|\ind(\gamma_{0,l})\cdot\exp(\chi)$
and hence $\ind(\gamma)|\ind(\gamma_{0,l})\cdot\exp(\chi)$.
\end{proof}
We will apply \citet[Theorem 5.1]{harbater2009applications}, which
states that $\ind(\gamma)=\lcm(\ind(\gamma_{U}),\ind(\gamma_{P}))$
for each $P\in\mathfrak{P}$. Since we already showed that $s$ can
never lower index of Brauer classes as in section \ref{sub:s-Splits-restriction},
we will be done if we could show that $\ind(\gamma)|\ind(\hat{\gamma});$
therefore it suffices to show that $\ind(\gamma_{U})|\ind(\hat{\gamma})$
and $\ind(\gamma_{P})|\ind(\hat{\gamma})$ for each $P\in\mathfrak{P}$,
respectively. We will deal with them in order. 

We start by recalling the notion of \emph{Azumaya algebras} and their
generalized Severi-Brauer varieties. The notion of a central simple
algebra over a field can be generalized to the notion of an \emph{Azumaya
algebra }over a domain $R$ (cf. \citet[Chapter 2]{saltman1999lectures},
or \citet[Part I, Section 1]{grothendieck1968legroupe}). The degree
of an Azumaya algebra $A$ over $R$ is the degree of $A\otimes_{R}F$
as a central simple algebra over the fraction field $F$ over $R$.
The \emph{Brauer group }of a domain $R$ is defined as the set of
equivalence classes of Azumaya algebras with the analogous operations,
where one replaces the vector spaces $V_{i}$ with projective modules
in the definition of Brauer equivalences. If $A$ is an Azumaya algebra
of degree $n$ over a domain $R$, and $1\le i<n$, there is a functorially
associated smooth projective $R$-scheme $\sb_{i}(A)$, called the
$i$-th \emph{generalized Severi-Brauer variety of $A$} (cf. \citet[p. 334]{bergh1988thebrauerseveri}).
For each $R$-algebra $S$, the $S$-points of $\sb_{i}(A)$ are in
bijection with the right ideals of $A_{S}=A\otimes_{R}S$ that are
direct summands of the $S$-module $A_{S}$ having dimension (i.e.
$S$-rank) $ni$. If $R$ is a field $F$, so that $A$ is a central
simple $F$-algebra, and if $E/F$ is a field extension, then $\sb_{i}(A)(E)\ne\phi$
if and only if $\ind(A_{E})$ divides $i$ (cf. \citet[Proposition 1.17]{knus1998thebook}).
Here $A_{E}\cong\mat_{m}(D_{E})$ for some $E$-division algebra $D_{E}$
and some $m\ge1$, and the right ideals of $E$-dimension $ni$ are
in natural bijection with the subspaces of $D_{E}^{m}$ of $D_{E}$-dimension
$i/\ind(A_{E})$ (cf. \citet[Proposition 1.12, Definition 1.9]{knus1998thebook}).
Thus the $F$-linear algebraic group $\gl_{1}(A)=\gl_{m}(D_{F})$
acts transitively on the points of the $F$-scheme $\sb_{i}(A)$.
We record \citet[Proposition 1.17]{knus1998thebook} here since we
will be using it over and over again in the sequel.
\begin{prop}
\label{pro:kmrt-sbv} Let $A$ be a central simple algebra over a
field $F$. The Severi-Brauer variety $\sb_{r}(A)$ has a rational
point over an extension $K/F$ if and only if the index $\ind(A_{K})$
divides $r$. In particular, $\sb(A)$ has a rational point over $K$
if and only if $K$ splits $A$. 
\end{prop}

\subsection{Index Computation Over Affine Open Set}

We compute $\ind(\gamma_{U})$ in this section; in particular, we
show that $\ind(\gamma_{U})|\ind(\hat{\gamma})$. Thanks to Lemma
\ref{lem:Indep-open-affine}, it suffices to show that there exists
an open affine subset $V\subset X$ so that $\ind(\gamma_{V})|\ind(\hat{\gamma})$
since we could replace $U$ by $V$ if necessary in the construction
we outlined in section \ref{sub:Construction-over-Open-Affine} and
this would not change $\gamma\in\H^{2}(K(\hat{X})$ by Lemma \ref{lem:Indep-open-affine}.
Therefore we will prove the following proposition, which shows that
there exists such an open affine subset $V$. 
\begin{prop}
Let $T$ be a complete discrete valuation ring. Let $\hat{X}$ be
a smooth projective $T$-curve with closed fibre $X$. Let $F$ be
the function field of $\hat{X}$ and $\hat{F}$ the completion of
$F$ with respect to $ $the discrete valuation determined by $X$.
Then for every $\hat{\gamma}\in\H^{2}(\hat{F})$, there exists an
affine open subset $V\subset X$ such that $\ind(\gamma_{V})|\ind(\hat{\gamma})$,
where $\gamma_{V}$ is the lift of $\hat{\gamma}$ to $F_{V}$ as
defined in section \ref{sub:Construction-over-Open-Affine}. \end{prop}
\begin{proof}
Recall that $\hat{\gamma}=\gamma_{0}+(\chi_{0,}t)\in\H^{2}(\hat{F})$
where $\gamma_{0}\in\H^{2}(k(X))$ and $\chi_{0}\in\H^{1}(k(X))$.
Therefore $\ind(\hat{\gamma})=\ind(\gamma_{0,l})\cdot\exp(\chi_{0})$,
where $l/k(X)$ is the field extension determined by $\chi_{0}$,
by \citet[Theorem 5.15]{jacob1990division}, since $\hat{F}$ is a
complete discretely valued field. 

Let $U$ be an open affine subset of $X$ such that neither $\gamma_{0}$
nor $\chi_{0}$ ramifies on any point of $U$. Recall that $\gamma_{U}=\tilde{\gamma}_{0}+(\tilde{\chi}_{0},t$)
where $\tilde{\gamma}_{0}\in\H^{2}(\hat{R}_{U})$ and $\tilde{\chi}_{0}\in\H^{1}(\hat{R}_{U})$.
Note that $\exp(\tilde{\chi}_{0})=\exp(\chi_{0})$ since $\H^{1}(\hat{R}_{U})\cong\H^{1}(k(X))$.
By Proposition \ref{pro:ele-index}, we have $\ind(\gamma_{U})|\ind(\tilde{\gamma}_{0,S})\cdot\exp(\tilde{\chi}_{0})$,
where $S/\hat{R}_{U}$ denotes the Galois cyclic extension determined
by $\tilde{\chi}_{0}$. Note when $V\subseteq U$, we have $\H^{r}(k[U])\subseteq\H^{r}(k[V])$
by purity, and hence $\H^{r}(\hat{R}_{U})\subseteq\H^{r}(\hat{R}_{V});$
so we have $\tilde{\gamma}_{0}\in\H^{2}(\hat{R}_{V})$ and $\tilde{\chi}_{0}\in\H^{1}(\hat{R}_{V})$.
Therefore it suffices to find some affine open subset $V\subset U$
such that $\ind(\tilde{\gamma}_{0,S'})|\ind(\gamma_{0,l})$, where
$S'/\hat{R}_{V}$ denotes the Galois cyclic extension determined by
$\tilde{\chi}_{0}$. 

Let $i=\ind(\gamma_{0,l})$ be the index of the restriction of $\gamma_{0}$
to $l$. Then Proposition \ref{pro:kmrt-sbv} implies that $\sb_{i}(\gamma_{0})(l)\ne\phi$;
in other words, there is an $l$-rational point in the $i$-th generalized
Severi-Brauer variety of $\gamma_{0}$. Hence the $\spec(k(X))$-morphism
$\pi:\sb_{i}(\gamma_{0})\times_{k(X)}l\to\spec(l)$ has a section
$\spec(l)\to\sb_{i}(\gamma_{0})\times_{k(X)}l$ over $\spec(k(X))$,
the generic point of the closed fibre $U$ of $\spec(\hat{R}_{U})$.
Choose a Zariski dense open subset $V\subseteq U$ such that this
section over $\spec(k(X))$ extends to a section over $V$, and such
that the image of this latter section lies in an open subset of $\sb_{i}(\gamma_{0})\times_{k(X)}l$
that is affine over $\hat{R}_{U}$. Then by Lemma \ref{lem:hhk-hensel},
the section over $V$ lifts to a section over $\spec(\hat{R}_{V})$,
thus we obtain an $L$-rational point of $\sb_{i}(\tilde{\gamma}_{0})\times_{\hat{R}_{V}}S'$,
where $L/F_{V}$ is the Galois cyclic extension determined by $\tilde{\chi}_{0}$;
or equivalently, $L$ is the fraction field of $S'$. This implies
that $\ind(\tilde{\gamma}_{0,S'})|i=\ind(\gamma_{0,l})$ by Proposition
\ref{pro:kmrt-sbv} again. 
\end{proof}

\subsection{Index Computation Over Closed Points}

It remains to show $\ind(\gamma_{P})|\ind(\hat{\gamma})$. This is
what we are going to do in this section. Note that $\gamma_{P}$ is
defined as $s_{P}\circ\res_{F_{\mathfrak{p}}}(\gamma_{U})$, where
$\res_{F_{\mathfrak{p}}}$ can only lower index of $\gamma_{U}$.
Since we have already shown that $\ind(\gamma_{U})|\ind(\hat{\gamma})$,
we have that $\ind(\gamma)$ will be completely determined by $\ind(\gamma_{U})$
if we could show that $\ind(\gamma_{\mathfrak{p}})$ does not go up
under the map $s_{P}$. Therefore we just need to show that $s_{P}$
cannot increase index of Brauer classes, or, $\ind(\gamma_{P})=\ind(s_{P}(\gamma_{\mathfrak{p}}))|\ind(\gamma_{\mathfrak{p}})$
.

We compute $\ind(\gamma_{\mathfrak{p}})$ first. Since $F_{\mathfrak{p}}$
is a complete discretely valued field, we have $\ind(\gamma_{\mathfrak{p}})=\ind((\gamma_{0,0}+(\chi_{1},\pi))_{M})\cdot\exp(\chi_{2}+(\pi^{r}))$,
where $M/\kappa(\mathfrak{p})$ is the Galois cyclic extension determined
by $\chi_{2}+(\pi^{r})\in\H^{1}(\kappa(\mathfrak{p}))$ by \citet[Theorem 5.15]{jacob1990division}.
It is not hard to compute $\ind((\gamma_{0,0}+(\chi_{1},\pi))_{M})$:
Since $M$ is a finite extension of $\kappa(\mathfrak{p})$, which
is a complete discretely valued field, we have that $M$ is a complete
discretely valued field as well. Let $e$ be the ramification index
of $M/\kappa(\mathfrak{p})$ and $\bar{M}$ the residue field of $M$.
Then by \citet[Exercise XII.3.2]{serre1979localfields}, $(\gamma_{0,0}+(\chi_{1,}\pi))_{M}=(\gamma_{0,0})_{\bar{M}}+(e\cdot\chi_{1},\pi')$,
where $\pi'$ is some uniformizer of $M$. Let $L/\kappa(\mathfrak{p})$
be the field extension determined by $e\cdot\chi_{1}$ and $\bar{L}$
the residue field of $L$. Then $\ind((\gamma_{0,0}+(\chi_{1,}\pi))_{M})=\ind((\gamma_{0,0})_{\bar{M}}+(e\cdot\chi_{1},\pi'))=\ind((\gamma_{0,0})_{\bar{M}\bar{L}})\cdot\exp(e\cdot\chi_{1})$. 

Now that we have an index formula for Brauer classes over $F_{\mathfrak{p}}$,
we are ready to show the following 
\begin{prop}
\label{pro:idx-compare}Let $T$ be a complete discrete valuation
ring. Let $\hat{X}$ be a smooth projective $T$-curve with closed
fibre $X$. Suppose that $U$ is an open affine subset of $X$ and
$P\in X\backslash U$ is a closed point. Let $\mathfrak{p}=(U,P)$
be the unique branch of $U$ at $P$ and let $\gamma_{P}$ and $\gamma_{\mathfrak{p}}$
be defined as above. Then we have $\ind(\gamma_{P})|\ind(\gamma_{\mathfrak{p}})$. \end{prop}
\begin{proof}
By Proposition \ref{pro:ele-index} we have that $\ind(\gamma_{P})|\ind((\tilde{\gamma}_{0,0}+(\tilde{\chi}_{1},\pi))_{\tilde{M}})\cdot\exp(\tilde{\chi}_{2}+(\pi^{r}))$,
where $\tilde{M}/F_{P}$ is the Galois cyclic extension determined
by $\tilde{\chi}_{2}+(\pi^{r})$. We claim that $\exp(\tilde{\chi}_{2}+(\pi^{r}))=\exp(\chi_{2})+(\pi^{r})$:
we have that $\exp(\tilde{\chi}_{2}+(\pi^{r}))=\lcm(\exp(\tilde{\chi}_{2}),\exp((\pi^{r})))$
and $\exp(\chi_{2}+(\pi^{r}))=\lcm(\exp(\chi_{2}),\exp((\pi^{r})))$.
Since $\exp(\tilde{\chi}_{2})=\exp(\chi_{2})$, we have proved that
$\exp(\tilde{\chi}_{2}+(\pi^{r}))=\exp(\chi_{2}+(\pi^{r}))$. Therefore
this proposition will follow if we can show that $\ind((\tilde{\gamma}_{0,0}+(\tilde{\chi}_{1},\pi))_{\tilde{M}})|\ind((\gamma_{0,0})_{\bar{M}}+(e\cdot\chi_{1},\pi'))$. 

Next we compute \begin{eqnarray*}
(\tilde{\gamma}_{0,0}+(\tilde{\chi}_{1},\pi))_{\tilde{M}} & = & (\tilde{\gamma}_{0,0})_{\tilde{M}}+(\tilde{\chi}_{1},\pi)_{\tilde{M}}\\
 & = & (\tilde{\gamma}_{0,0})_{\tilde{M}}+((\tilde{\chi}_{1})_{\tilde{M}},\pi)\\
 & = & (\tilde{\gamma}_{0,0})_{\tilde{M}}+((\tilde{\chi}_{1})_{\tilde{M}},(\pi')^{e})\\
 & = & (\tilde{\gamma}_{0,0})_{\tilde{M}}+(e\cdot(\tilde{\chi}_{1})_{\tilde{M}},\pi')\end{eqnarray*}

By Proposition \ref{pro:ele-index} again we immediately see that
$\ind((\tilde{\gamma}_{0,0}+(\tilde{\chi}_{1},\pi))_{\tilde{M}})|\ind((\tilde{\gamma}_{0,0})_{\tilde{M}\tilde{L}})\cdot\exp(e\cdot(\tilde{\chi}_{1})_{\tilde{M}})$,
where $\tilde{L}/F_{P}$ denotes the Galois cyclic extension determined
by $e\cdot\tilde{\chi}_{1}$. Clearly $\exp(e\cdot(\tilde{\chi}_{1})_{\tilde{M}})|\exp(e\cdot(\chi_{1}))$,
so we will be done if we can show that $\ind((\tilde{\gamma}_{0,0})_{\tilde{M}\tilde{L}})|\ind((\gamma_{0,0})_{\bar{M}\bar{L}})$,
which we will do in the following Lemma \ref{lem:index-core}.\end{proof}
\begin{lem}
\label{lem:index-core}In line with the notation in \ref{pro:idx-compare},
we have that $\ind((\tilde{\gamma}_{0,0})_{\tilde{M}\tilde{L}})|\ind((\gamma_{0,0})_{\bar{M}\bar{L}})$. \end{lem}
\begin{proof}
Let $\tilde{M}'/F_{P}$ be the Galois cyclic extension determined
by $\chi_{2}$. Clearly it suffices to prove that $\ind((\tilde{\gamma}_{0,0})_{\tilde{M}'\tilde{L}})|\ind((\gamma_{0,0})_{\bar{M}\bar{L}})$
since $\ind((\tilde{\gamma}_{0,0})_{\tilde{M}\tilde{L}})|\ind((\tilde{\gamma})_{\tilde{M}'\tilde{L}})$.
Let $i=\ind((\gamma_{0,0})_{\bar{M}\bar{L}})$. By Proposition \ref{pro:kmrt-sbv},
we have that $\sb_{i}(\gamma_{0,0})(\bar{M}\bar{L})\ne\phi$, or equivalently,
the morphism $\sb_{i}(\gamma_{0,0})\times_{\kappa(P)}\bar{M}\bar{L}$
has a section $\spec(\bar{M}\bar{L})\to\sb_{i}(\gamma_{0,0})\times_{\kappa(P)}\bar{M}\bar{L}$.
By Lemma \ref{lem:hhk-hensel}, this section lifts to a section over
$\spec(\hat{R}_{P})$; thus we obtain a $\tilde{M}'\tilde{L}$-rational
point of $\sb_{i}(\tilde{\gamma}{}_{0,0})\times_{\hat{R}_{P}}S$(note
that $\gamma_{0,0}\in\H^{2}(\hat{R}_{P})$), where $S$ is the integral
closure of $\hat{R}_{P}$ in $\tilde{M}'\tilde{L}$; or equivalently,
a $\tilde{M}'\tilde{L}$-rational point of $\sb_{i}(\tilde{\gamma}{}_{0,0})\times_{F_{P}}\tilde{M}'\tilde{L}$.
Therefore $\ind((\tilde{\gamma}_{0,0})_{\tilde{M}'\tilde{L}})|i$
again by Proposition \ref{pro:kmrt-sbv}, which proves this lemma. 
\end{proof}
The following Corollary is immediate: 
\begin{cor}
\label{cor:preserve-index}The homomorphism $s:\H^{2}(\hat{F})\to\H^{2}(F)$
preserves index of Brauer classes. \end{cor}
\begin{proof}
This is simply Corollary \ref{cor:raise-index} plus Proposition \ref{pro:idx-compare}.
\end{proof}

\section{Indecomposable and noncrossed product Division Algebras over Curves
over complete Discrete Valuation Rings \label{sec:Indecomposable-and-Noncrossed-Product}}

Let $T$ be a complete discrete valuation ring. Let $\hat{X}$ be
a smooth projective $T$-curve with closed fibre $X$. Let $F$ be
the function field of $\hat{X}$ and $\hat{F}$ the completion of
$F$ with respect to $ $the discrete valuation determined by $X$.
We construct indecomposable division algebras and noncrossed product
division algebras over $F$ of prime power index for all primes $q$
where $q$ is different from the characteristic of the residue field
of $T$. Note that the existence of such algebras are already known
when residue field of $T$ is a finite filed, cf. \citet{brussel2009indecomposable}.
 Our construction here is almost identical to \citet[Section 4]{brussel2009indecomposable},
we list it here for the reader's convenience.

\subsection{Indecomposable Division Algebras over $F$}

First we recall the construction of indecomposable division algebras
over $\hat{F}$, this is done in \citet[Proposition 4.2]{brussel2009indecomposable}.
\begin{prop}
\label{pro:ind-completion} Let $T$ be a complete discrete valuation
ring and let $\hat{X}$ be a smooth projective curve over $\spec(T)$
with closed fibre $X$. Let $F$ be the function field of $\hat{X}$
and $\hat{F}$ the completion of $F$ with respect to the discrete
valuation induced by $X$. Let $e,i$ be integers satisfying $l\le e\le2e-1$.
For any prime $q\ne\chara(k)$, there exists a Brauer class $\hat{\gamma}\in\H^{2}(\hat{F})$
satisfying $\ind(\hat{\gamma})=q^{i},\exp(\hat{\gamma})=q^{e}$ and
whose underlying division algebra is indecomposable.
\end{prop}
Then we lift $\hat{\gamma}$ to $F$ by using the splitting map $s$
we defined in section \ref{sec:Splitting-Map}, and show that the
lift is in fact indecomposable. 
\begin{thm}
In the notation of Theorem \ref{pro:ind-completion}. Then there exists
an indecomposable division algebra $D$ over $F$ such that $\ind(D)=q^{i}$
and $\exp(D)=q^{e}$ .\end{thm}
\begin{proof}
By Proposition \ref{pro:ind-completion}, there exists $\hat{\gamma}\in\br(\hat{F})$
with $\ind(\hat{\gamma})=q^{i}$ and $\exp(\hat{\gamma})=q^{e}$ and
whose underlying division algebra is indecomposable. By Corollary
\prettyref{cor:preserve-index}, $\gamma=s(\hat{\gamma})$ has index
$q^{i}$ too. Since $s$ splits the restriction map, we have $\exp(\gamma)=q^{e}$.
We show the division algebra underlying $\gamma$ is indecomposable.

We proceed by contradiction. Assume $\gamma=\beta_{1}+\beta_{2}$
represents a nontrivial decomposition, then $\hat{\gamma}=\res_{\hat{F}}(\beta_{1})+\res_{\hat{F}}(\beta_{2})$.
Since the index can only go down under restriction, we have that $\ind(\hat{\gamma})=\ind(\res_{\hat{F}}(\beta_{1}))\cdot\ind(\res_{\hat{F}}(\beta_{2}))$,
which represents a nontrivial decomposition of the division algebra
underlying $\hat{\gamma}$, a contradiction. 
\end{proof}

\subsection{Noncrossed Products over $F$ }

Again we will construct noncrossed product division algebras over
$\hat{F}$ and use the splitting map $s$ to lift it to $F$ and show
that the lift represents a noncrossed product division algebra over
$F$. 

The construction over $\hat{F}$ is in line with \citet{brussel1995noncrossed}
where noncrossed products over $Q(t)$ and $Q((t))$ are constructed.
In order to mimic the construction in \citet{brussel1995noncrossed},
we need only note that both $ $Chebotarev density theorem and the
Gruwald-Wang theorem hold for global fields which are characteristic
$p$ function fields. Then the arguments in \citet{brussel1995noncrossed}
apply directly to yield noncrossed products over $\hat{K}(\hat{X})$
of index and exponent given below:

The following is \citet[Theorem 4.7]{brussel2009indecomposable}.
\begin{thm}
\label{thm:nc-over-compl}Let $T$ be a complete discrete valuation
ring with residue field $k$ and let $\hat{X}$ be a smooth projective
curve over $\spec(T)$. Let $F$ be the function field of $\hat{X}$
and let $\hat{F}$ be the completion of $F$ with respect to the discrete
valuation induced by the closed fibre. For any positive integer $a$,
let $\epsilon_{a}$ be a primitive $a$-th root of unity. Set $r$
and $s$ to be maximum integers such that $\mu_{q^{r}}\subset k(X)^{\times}$
and $\mu_{q^{s}}\subset k(X)(\epsilon_{q^{r+1}})$. Let $n,m$ be
integers such that $n\ge1,n\ge m$ and $n,m\in{r}\cup[s,\infty)$.
Let $a,l$ be integers such that $l\ge n+m+1$ and $0\le a\le1-n$.
(See \citet[Page 384-385]{brussel1995noncrossed} for more information
regarding these constraints.) Let $q\ne\chara(k)$ be a prime number.
Then there exists noncrossed product division algebras over $\hat{F}$
with index $q^{l+1}$ and exponent $q^{l}$. \end{thm}
\begin{cor}
Let $R,k,\hat{X},X,F,\hat{F},q,a,l$ be as in Theorem \ref{thm:nc-over-compl}.
Then there exists noncrossed product division algebras over $F$ of
index $q^{l+a}$ and exponent $q^{l}$. \end{cor}
\begin{proof}
Let $\hat{\gamma}$ be the Brauer class representing a noncrossed
product over $\hat{F}$ of index $q^{l+a}$ and exponent $q^{l}$.
Let $D$ be the division algebra underlying the Brauer class $s(\hat{\gamma})$.
By Corollary \ref{cor:preserve-index}, we know that $\ind(D)=\ind(\hat{\gamma})$. 

Assume that $D$ is a crossed product with maximal Galois subfield
$M/F$. Then $M\hat{F}$ splits $\hat{\gamma}$, is of degree $\ind(\hat{\gamma})$
and is Galois. This contradicts the fact that $\hat{\gamma}$ is a
noncrossed product. 
\end{proof}

\section*{References}

\bibliographystyle{plainnat}
\bibliography{lytez}

\end{document}